\documentclass[11pt,a4paper]{amsart}
\usepackage[T1]{fontenc}
\usepackage{lmodern}
\usepackage[margin=29mm]{geometry}
\usepackage{amsmath,amssymb,mathtools}
\usepackage{stmaryrd}
\usepackage{dsfont}
\usepackage{mathrsfs}
\usepackage[expansion=false]{microtype}
\usepackage{xcolor}
\usepackage{booktabs}
\usepackage[colorlinks=true,linkcolor=blue!45!black,citecolor=blue!45!black,urlcolor=blue!45!black]{hyperref}
\usepackage{enumitem}
\setlist[enumerate]{label=\textup{(\roman*)},leftmargin=*,itemsep=3pt}
\allowdisplaybreaks[1]
\renewcommand{\geq}{\geqslant}
\renewcommand{\leq}{\leqslant}
\renewcommand{\le}{\leqslant}
\renewcommand{\ge}{\geqslant}

\newtheorem{theorem}{Theorem}[section]
\newtheorem{proposition}[theorem]{Proposition}
\newtheorem{lemma}[theorem]{Lemma}
\newtheorem{corollary}[theorem]{Corollary}
\theoremstyle{definition}

\newtheorem{example}[theorem]{Example}
\newtheorem{remark}[theorem]{Remark}

\newcounter{proofcount}
\AtBeginEnvironment{proof}{\stepcounter{proofcount}}
\newtheorem{claim}{Claim}
\newtheorem*{claim*}{Claim}
\makeatletter                  
\@addtoreset{claim}{proofcount}
\makeatother                   

\newenvironment{claimproof}[1][Proof of Claim \theclaim.]
 {%
	\proof[#1]%

}
{%
	\endproof%
}

\newcommand{\Q}{\mathbb Q}
\newcommand{\Z}{\mathbb Z}
\newcommand{\N}{\mathbb N}
\newcommand{\alg}{\mathrm{alg}}
\newcommand{\sep}{\mathrm{sep}}

\newcommand{\bd}{\boldsymbol d}
\DeclareMathOperator{\Gal}{Gal}
\DeclareMathOperator{\Aut}{Aut}
\DeclareMathOperator{\Fix}{Fix}
\DeclareMathOperator{\Th}{Th}

\DeclareMathOperator{\charac}{char}
\DeclareMathOperator{\id}{id}

\newcommand{\ACFA}{\mathrm{ACFA}}
\title{PAC fields and difference fields with generic automorphisms}
\author[\"{O}.~Beyarslan]{\"{O}zlem Beyarslan}
\address{Bo\u{g}azi\c{c}i \"{U}niversitesi}
\email{ozlem.beyarslan@boun.edu.tr}
\thanks{The first author acknowledges support from the
T\"{u}bitak 1001 grant no.~124F359}

\author[P.~Kowalski]{Piotr Kowalski}
\address{Instytut Matematyczny\\
Uniwersytet Wroc{\l}awski\\
Wroc{\l}aw\\
Poland}
\email{pkowa@math.uni.wroc.pl}
\urladdr{http://www.math.uni.wroc.pl/~pkowa/}
\thanks{The second author acknowledges support from the
Narodowe Centrum Nauki grant no.~2018/31/B/ST1/00357
and the T\"{u}bitak 1001 grant no.~124F359.}
\date{}


\begin{document}
\begin{abstract}
We study the existence of model companions after adjoining an
automorphism to field theories in two complementary settings.
For bounded PAC fields of
characteristic zero containing all roots of unity and admitting a
cyclic Galois quotient of prime order, a suitable choice of countably many constants gives a
model-complete field theory which has no model
companion after adding an automorphism fixing the named constants. This applies, in particular, when the absolute Galois group
is free profinite or free pro-$p$ of finite positive rank. Secondly, for every completion of $\ACFA$ in
every characteristic, adjoining a second commuting automorphism has
no model companion. The two explicit obstructions use a common
ordered product of inertia automorphisms along an infinite orbit,
combined respectively with an equivariant PAC completion and an
algebraically closed difference-field embedding.
\end{abstract}

\maketitle
\tableofcontents

\section{Introduction}\label{sec:introduction}

The existence of $\ACFA$, the model companion of the theory of
fields with an automorphism, naturally raises the question of
adjoining further automorphisms. Hrushovski showed in the 1990s that the theory of fields
with two commuting automorphisms has no model companion;
accounts of his argument appear in
\cite[Section~3]{KikyoSurvey} and
\cite[Section~1.11]{Survey}. 
Chatzidakis and Pillay asserted in the last two bullet points
of \cite[Section~3.11]{CP} that analogous nonexistence results
hold for $\ACFA$ and for the theory of pseudofinite fields
expanded by an automorphism.

Chatzidakis subsequently gave a variant of Hrushovski's proof
in \cite[Section~1.11]{Survey}, but emphasized in
\cite[Remark~1.12]{Survey} that this argument does not by itself
settle the questions for pseudofinite fields or completions of
$\ACFA$. The difficulty is to realize the required difference
equations in an extension whose reduct still satisfies the
prescribed theory. Our proofs follow Chatzidakis's strategy, comparing an
existential condition forbidden at certain finite periods with
its realization along an infinite orbit.
 In the $\ACFA$ case, the construction requires
extending two commuting automorphisms to an algebraically closed
field while preserving commutativity. In the pseudofinite case,
the corresponding construction must also respect the procyclic
Galois action.

We now state the two main results. Recall that a field is
\emph{PAC} if every absolutely irreducible variety over it has
a rational point. A field $F$ is \emph{bounded} if it has only
finitely many finite separable extensions of each degree inside
a fixed separable closure $F^{\sep}$. We write $\Gal(F)$ for the absolute
Galois group of $F$.
For a prime $\ell$, let $C_\ell$ denote the cyclic group
of order $\ell$.

\begin{theorem}[No bounded PACA]\label{b:thm:main}
Let $F$ be a bounded PAC field of characteristic zero containing
all roots of unity. Suppose that, for some prime $\ell$, there is
a continuous surjection
\[
\Gal(F)\twoheadrightarrow C_\ell.
\]
Then there is a countable tuple $\bd$ from $F$ such that
$\Th(F,\bd)$
is model-complete and the theory
\[
\Th(F,\bd)\cup
\{\sigma\text{ is a field automorphism and}\ \sigma(\bd)=\bd\}
\]
has no model companion.
\end{theorem}

Under the stated field
hypotheses, the theorem includes absolute Galois groups that
are free profinite or free pro-$p$ of finite positive rank;
see Corollary~\ref{b:cor:free}.

The pseudofinite case, corresponding to
$\Gal(F)\cong\widehat{\Z}$ and included here under the same
hypotheses, is quite special and will be treated separately
in Feyza Nur Okudan's PhD thesis.

Our second result concerns an additional automorphism of a
model of a fixed completion of $\ACFA$.

\begin{theorem}[No ACFAA]
\label{a:thm:main}
Let $T$ be a completion of $\ACFA$ (arbitrary characteristic) and let $\sigma$ be its
distinguished automorphism. Then the theory
\[
T\cup
\{\tau\text{ is a field automorphism and}\
  \tau\sigma=\sigma\tau\}
\]
has no model companion in the language expanded by a
new function symbol $\tau$.
\end{theorem}
For a theory $T$, we write $T_\tau$ for the theory of models
of $T$ equipped with an automorphism $\tau$. Our results support the conjecture that $T_\tau$ has no model
companion whenever $T$ is unstable, complete, and model-complete
\cite[Introduction]{KikyoSurvey}. Kikyo proved this assuming
amalgamation for $T_\tau$ \cite[Theorem~3.3]{Kikyo}.
Failure of amalgamation for $\ACFA_\tau$ is reported in
\cite[Section~4]{Kikyo} and \cite[Introduction]{KikyoShelah},
but these references do not establish it for every completion
in every characteristic. Chatzidakis and Pillay asked whether the theory of pseudofinite
fields has the amalgamation property for automorphisms
(PAPA) \cite[ end of Section~3]{CP}.
Kikyo and Shelah subsequently stated that it was unknown whether
the theory of pseudofinite fields with an automorphism has
the amalgamation property \cite[Introduction]{KikyoShelah}.
A related difficulty concerning the existence of model companions
for complete theories of pseudofinite fields with an automorphism
is discussed in \cite[Remark~1.12]{Survey}.

The paper is organized as follows.
In Section~\ref{sec:tools}, we develop the common algebraic tools.
In Section~\ref{sec:chatzidakis-argument}, we recall Chatzidakis's
variant of Hrushovski's argument and explain the proof strategy,
including the requirement that realizing extensions preserve
the prescribed reduct theory.
In Sections~\ref{sec:bounded} and~\ref{sec:acfa}, we prove
Theorems~\ref{b:thm:main} and~\ref{a:thm:main}, respectively.
In Section~\ref{sec:mc-reducts}, we discuss the transfer of model
companions to completions and the
relation between model-completeness and the existence of
weakly saturated existentially closed models.

\medskip
\noindent
\textbf{Use of AI.}
The authors used ChatGPT (GPT-6 Astra in Pro mode) to assist in identifying examples, developing proofs, and drafting portions of the text. The authors subsequently refined and streamlined the arguments and assume full responsibility for the final content.

\section{Common conventions and algebraic tools}\label{sec:tools}
This section collects the algebraic tools used in the subsequent
arguments. We first recall the required facts about 
Kummer extensions, Puiseux and Hahn series, and ramification.
We then construct commuting lifts of automorphisms, describe the
behavior of inertia under restriction and conjugation, and prove
a convergence result for ordered products in profinite groups.
Together, these tools allow us to prescribe the action on a chosen root in a Kummer extension while preserving commutation with another automorphism,
providing the algebraic input for the realization arguments in
Sections~\ref{sec:chatzidakis-argument}, ~\ref{sec:bounded}, and~\ref{sec:acfa}.

\subsection{Preliminaries from field theory and valuation theory}
We collect here necessary classical algebraic facts focusing on Kummer theory and ramification theory.

Let $F$ be a field. For every positive integer $n$, we denote
\[
\mu_n(F):=\{a\in F^\times\mid a^n=1\},
\qquad
\mu_\infty(F):=\bigcup_{n\ge1}\mu_n(F).
\]
Thus $\mu_\infty(F)$ is the multiplicative group of all
roots of unity in $F$. In characteristic zero, we use
the abbreviations
\[
\mu_n:=\mu_n(\Q^{\alg}),
\qquad
\mu_\infty:=\mu_\infty(\Q^{\alg}),
\]
where $\Q^{\alg}$ is a fixed algebraic closure of $\Q$.

A field is \emph{bounded} if it has only finitely many
separable extensions of each finite degree inside a fixed
separable closure. Equivalently, its absolute Galois group
is small, that is, it has only finitely many open subgroups
of each finite index.

We recall two facts from Kummer theory. Let $K$ be a field,
let $\ell\ne\operatorname{char}(K)$ be prime, and suppose
that $K$ contains a primitive $\ell$-th root of unity $\omega$.

\begin{enumerate}[label=\textup{(\roman*)}]
\item
If $c\in K^\times\setminus(K^\times)^\ell$ and $\beta^\ell=c$,
then $K(\beta)/K$ is cyclic of degree $\ell$, with automorphisms
\[
\beta\longmapsto\omega^j\beta,
\qquad j\in\mathbb Z/\ell\mathbb Z.
\]
Conversely, every cyclic extension of $K$ of degree $\ell$
is obtained in this way and
\[
\operatorname{Gal}(K)\twoheadrightarrow C_\ell
\quad\Longleftrightarrow\quad
K^\times/(K^\times)^\ell\ne\{1\},
\]
where $\operatorname{Gal}(K)$ denotes the absolute Galois group
and the surjection is required to be continuous
\cite[Section~9.24, Tag~09I6]{Stacks}.

\item
Suppose that the classes of $f_1,\ldots,f_r\in K^\times$
are linearly independent in the $\mathbb F_\ell$-vector space
$K^\times/(K^\times)^\ell$. Choose $\alpha_i$ with
$\alpha_i^\ell=f_i$ and put $E=K(\alpha_1,\ldots,\alpha_r)$.
Then
\[
[E:K]=\ell^r,
\qquad
\operatorname{Gal}(E/K)\cong(C_\ell)^r,
\]
and the monomials
\[
\alpha_1^{e_1}\cdots\alpha_r^{e_r},
\qquad 0\le e_i<\ell,
\]
form a $K$-basis of $E$. Every assignment
$\alpha_i\mapsto\omega^{j_i}\alpha_i$ therefore defines
a $K$-automorphism of $E$. For an infinite independent family $(f_i)_{i\in I}$,
the same assertions apply to every finite subfamily, and arbitrary
choices of the multipliers define an automorphism of the field
generated by all the radicals. More generally, if
$\gamma\in\operatorname{Aut}(K)$ and
$\gamma(f_i)=f_{\pi(i)}$ for a permutation $\pi$ of $I$,
then any assignments
\[
\alpha_i\longmapsto\xi_i\alpha_{\pi(i)},
\qquad \xi_i\in\mu_\ell(K),
\]
extend $\gamma$ to an automorphism of that field.
These are consequences of the Kummer correspondence
\cite[Theorem~5.30 and Remark~5.32]{MilneFT}.
\end{enumerate}
Let $k$ be a field and $x$ an indeterminate. The
\emph{Puiseux series field} is
\[
P(k,x):=\bigcup_{m\ge1}k((x^{1/m})),
\]
where the fractional powers are chosen compatibly.
The \emph{Hahn series field}
\[
H(k,x):=k((x^{\Q}))
\]
consists of formal sums $\sum_{r\in\Q}a_rx^r$ with
$a_r\in k$ and well-ordered support
$\{r\in\Q\mid a_r\ne0\}$, equipped with the usual series
operations. There is a natural embedding
$P(k,x)\hookrightarrow H(k,x)$ whose image consists
precisely of the series with support contained in
$\frac1m\Z$ for some $m\ge1$.
If $k$ is algebraically closed of characteristic zero,
the Newton--Puiseux theorem says that $P(k,x)$ is an
algebraic closure of the field of Laurent series $k((x))$.
In characteristic $p>0$, however, $P(k,x)$ is not
algebraically closed: the equation $Y^p-Y=x^{-1}$
has the Hahn-series solution
\[
\sum_{n\ge1}x^{-1/p^n}\in H(k,x)\setminus P(k,x);
\]
see \cite[Section~1]{Poonen}. For algebraically closed $k$ in any characteristic,
$H(k,x)$ is algebraically closed, since its value group
$\Q$ is divisible \cite[Corollary~4]{Poonen}.

Let $k$ be algebraically closed, put $K=k(x)$, and fix
$\Omega=K^{\alg}$. For $a\in k$, the \emph{point valuation}
$v_a=\operatorname{ord}_{x=a}$ is normalized by $v_a(x-a)=1$.
At infinity use $v_\infty=\operatorname{ord}_{u=0}$, where $u=x^{-1}$.
Their valuation rings are $k[x]_{(x-a)}$ and $k[u]_{(u)}$,
respectively, and their residue fields are $k$.

For a valuation $w$ on a field $E$ in general, we write
\[
\mathcal O_w=\{z\in E:w(z)\ge0\},\qquad
\mathfrak m_w=\{z\in E:w(z)>0\},\qquad
\kappa(w)=\mathcal O_w/\mathfrak m_w,
\]
with $w(0)=\infty$. If $K\subseteq E$ is a field extension, then a \emph{prolongation} of a valuation $v$ of
$K$ to $E$ is a valuation $w$ with
$\mathcal O_w\cap K=\mathcal O_v$. For a finite extension $E/K$, let $w$ prolong a point
valuation $v$ of $K$, and normalize $w$ so that
$w(E^\times)=\Z$.
The \emph{ramification index} is the positive integer
\[
e(w/v):=[w(E^\times):w(K^\times)].
\]
Since $v(K^\times)=\Z$, this means that
\[
w(f)=e(w/v)\,v(f)\qquad\text{for all $f\in K^\times$}.
\]
In particular, at a finite point $a$ we have
$e(w/v_a)=w(x-a)$.

We often just say ``point'' for ``point valuation''. Since the residue fields above a point are finite extensions of the
algebraically closed field $k$, they are again $k$. In this setting
the point is \emph{unramified} in $E/K$ if every prolongation has
$e(w/v)=1$, and \emph{ramified} otherwise. It is \emph{totally
ramified} if there is a unique prolongation and its ramification
index is $[E:K]$.

Every valuation on a field admits a prolongation to any
algebraic field extension \cite[Tag~00IA]{Stacks}.
In an algebraic extension, the induced residue-field extension
is algebraic; this follows from the finite case by passing
to finite subextensions.
Consequently, every prolongation of a point valuation of $K$
to $\Omega$ has residue field $k$, since $k$ is algebraically
closed. 

Choose a prolongation of a point valuation $v$ on $K$ to $\widetilde v$ on $\Omega$.
For $G=\Gal(\Omega/K)$, its \emph{decomposition group} and
\emph{inertia group} are
\[
D_{\widetilde v}
=\{g\in G:g(\mathcal O_{\widetilde v})
                         =\mathcal O_{\widetilde v}\},\qquad
I_{\widetilde v}
=\ker\!\left(D_{\widetilde v}\to
               \Aut(\kappa(\widetilde v)/k)\right).
\]
The same definitions apply to a finite Galois extension $E/K$,
using a prolongation $w$ on $E$ and $G=\Gal(E/K)$.
An \emph{inertia automorphism above $a$} is an element of
$I_{\widetilde v}$ for a chosen prolongation of $v_a$.

For a finite Galois extension $E/K$, the group $\Gal(E/K)$
acts transitively on the prolongations of $v_a$ to $E$,
with $g$ sending $w$ to $w\circ g^{-1}$.
These prolongations correspond to the points above $a$
on the normalization of $\mathbb P^1_k$ in $E$, that is the
nonsingular projective curve $C$ with function field $k(C)=E$,
together with the finite morphism $\pi:C\longrightarrow\mathbb P^1_k$
corresponding to the inclusion $k(x)\subseteq E$.

The degree formula for the finite normalization
of the curve gives
\[
[E:K]=\sum_{w\mid v_a}e(w/v_a)[\kappa(w):k].
\]
All residue degrees are $1$, and the ramification indices are equal.
Consequently,
\[
|I_w(E/K)|=|D_w(E/K)|=e(w/v_a).
\]
In particular, inertia at an unramified point is trivial.

Let $E/K$ be a finite Galois extension inside $\Omega$,
let $\widetilde v$ be a prolongation to $\Omega$ of a point
valuation $v$ of $K$, and let $w$ be its restriction to $E$.
Restriction of automorphisms induces surjective homomorphisms
\[
D_{\widetilde v}\twoheadrightarrow D_w(E/K),
\qquad
I_{\widetilde v}\twoheadrightarrow I_w(E/K)
\]
\cite[Tag~0BSX]{Stacks}.
In particular, every element of the finite inertia group
$I_w(E/K)$ lifts to an element of $I_{\widetilde v}$.

For a monic polynomial $P(T)\in K[T]$ of degree $n$, with roots
$\alpha_1,\ldots,\alpha_n$ in an algebraic closure of $K$, counted
with multiplicity, its \emph{discriminant} is
\[
\operatorname{disc}(P)
:=\prod_{1\leq i<j\leq n}(\alpha_i-\alpha_j)^2.
\]
We shall use the following standard discriminant criterion
\cite[Tags~0BJF, 0EXR, and~0BTF]{Stacks}.
Let $(K,v)$ be a discretely valued field with valuation ring
$\mathcal O_v$, let $P\in\mathcal O_v[T]$ be monic and separable,
and let $E/K$ be its splitting field.
If $v(\operatorname{disc}(P))=0$, then
$I_w(E/K)=\{1\}$ for every prolongation $w$ of $v$ to $E$,
and hence $E/K$ is unramified at $v$.

\subsection{Some algebraic results}

We need the following result about commuting lifts which will be used in the proof of Claim 1 from the proof of Lemma \ref{b:lem:infinite} (corresponding to (A) from Section
\ref{sec:chatzidakis-argument}). The following construction is a variant of the commuting
actions on generalized power series used in
\cite[Theorem~7]{BHAlgClosure}; see also
\cite[Section~2.1]{BCGeometric}.
\begin{lemma}\label{a:lem:hahn}
Let $k$ be algebraically closed of characteristic zero,
and let $\rho,(\gamma_i)_{i\in I}\in\Aut(k)$ satisfy
$\rho\gamma_i=\gamma_i\rho$ for every $i$.
Let $x$ be an indeterminate, put $A:=P(k,x)$, and let
$\Omega$ be the relative algebraic closure of $k(x)$ in $A$.
For any character $\epsilon:\Q\to k^\times$ whose image
is fixed pointwise by every $\gamma_i$, the formulas
\[
s\!\left(\sum_r a_rx^r\right)
   =\sum_r\rho(a_r)\epsilon(r)x^r,
\qquad
d_i\!\left(\sum_r a_rx^r\right)
   =\sum_r\gamma_i(a_r)x^r
\]
define automorphisms of $A$ extending $\rho,\gamma_i$
and satisfying $sd_i=d_is$.
They preserve $\Omega$, which is an algebraic closure
of $k(x)$; their restrictions are denoted by the same letters.

If every $\gamma_i$ fixes $\mu_\infty(k)$ pointwise,
then $\epsilon$ may be chosen with values in $\mu_\infty(k)$
as follows:
\begin{enumerate}[label=\textup{(\roman*)}]
\item For any prescribed $\zeta\in\mu_\infty(k)$,
one can arrange $\epsilon(1)=\zeta$, giving $s(x)=\zeta x$.
\item If $\rho=\mathrm{id}_k$, then for any prime $\ell$
one can arrange $\epsilon(1)=1$ and
$\epsilon(1/\ell^m)$ of exact order $\ell^m$ for every $m\ge1$.
Then $s$ fixes $k(x)$ pointwise, and
$b_m:=x^{1/\ell^m}\in\Omega$ is fixed by every $d_i$
and has exact $s$-period $\ell^m$.
\end{enumerate}
\end{lemma}

\begin{proof}
Newton--Puiseux makes $A$ algebraically closed, hence
$\Omega$ is an algebraic closure of $k(x)$.
The character identity makes $s$ multiplicative;
its inverse replaces each coefficient $a_r$ by
$\rho^{-1}(a_r\epsilon(r)^{-1})$, while $d_i^{-1}$
acts coefficientwise by $\gamma_i^{-1}$.
The hypotheses give $sd_i=d_is$.
Since $s(x)=\epsilon(1)x$ and $d_i(x)=x$, these maps
preserve $k(x)$ and hence $\Omega$.

For the additional assertions, $\mu_\infty(k)$ is divisible
and hence injective as an abelian group, so every character
from a subgroup of $\Q$ to $\mu_\infty(k)$ extends to $\Q$.
Extending $n\mapsto\zeta^n$ from $\Z$ proves~\textup{(i)}.
For~\textup{(ii)}, choose primitive $\ell^m$-th roots $\xi_m$
with $\xi_0=1$ and $\xi_{m+1}^{\ell}=\xi_m$, and extend
the character $a/\ell^m\mapsto\xi_m^a$ from $\Z[1/\ell]$.
Then $\epsilon(1)=1$, and, since $\rho=\mathrm{id}_k$,
\[
d_i(b_m)=b_m,\qquad s^n(b_m)=\xi_m^n b_m.
\]
Thus each $b_m$ has exact period $\ell^m$.
\end{proof}
\begin{remark}\label{pposit}
Lemma~\ref{a:lem:hahn} holds in arbitrary characteristic
after replacing $A$ by the Hahn field $H:=k((x^{\Q}))$
and requiring $\ell\ne\charac(k)$ in part~\textup{(ii)}.
Indeed, $H$ is algebraically closed
\cite[Corollary~4]{Poonen}, and the same formulas preserve
supports and define automorphisms preserving the relative
algebraic closure of $k(x)$ in $H$.
The character constructions remain valid because
$\mu_\infty(k)$ is divisible and, for
$\ell\ne\charac(k)$, coherent primitive $\ell^m$-th
roots exist for all $m$.
\end{remark}
We need the following results about inertia and finite ramification which will be used in the proof of Claim 2 from Lemma \ref{b:lem:infinite}.
\begin{lemma}\label{a:lem:ramification}
Let $k$ be algebraically closed, $x$ be transcendental over $k$, $K=k(x)$, and $\Omega=K^{\alg}$.
\begin{enumerate}[label=\textup{(\roman*)}]
\item Every finite Galois extension $E/K$ inside $\Omega$ ramifies at only finitely many points of $\mathbb P^1(k)$.
\item If a finite Galois extension $E/K$ inside $\Omega$
is unramified at $a\in k$, then every inertia automorphism
in $\Gal(\Omega/K)$ above $a$ fixes $E$ pointwise.
\item If $a\neq d$ and $\ell\neq\charac(k)$ is prime, then
\[
 K(y)/K,\qquad y^\ell=\frac{x-a}{x-d},
\]
is cyclic of degree $\ell$, totally ramified at $a,d$, and unramified elsewhere.

Moreover, for a prescribed primitive $\ell$-th root $\zeta\in k$, an inertia automorphism above $a$ can be chosen to send $y$ to $\zeta y$.

\item Let $s\in\Aut(\Omega)$ satisfy
\[
s(k)=k,\qquad s(x)=\zeta x, \qquad s(\zeta)=\zeta
\]
 for some
$\zeta\in k^\times$.
If $g$ is an inertia automorphism above $a\in k$,
then, for every $j\in\Z$, the conjugate $s^jgs^{-j}$
is an inertia automorphism above $\zeta^{-j}s^j(a)$.
\end{enumerate}
\end{lemma}

\begin{proof}
For (i), choose an integral primitive element with monic minimal polynomial $P\in k[x][Y]$. Its discriminant is a nonzero polynomial in $x$. All finite ramified points lie among its zeros; there is only one additional possible point, infinity.

For (ii), restriction sends inertia into inertia.
Since $E/K$ is unramified at $a$, its inertia groups above
$a$ are trivial. Hence every inertia automorphism in
$\Gal(\Omega/K)$ above $a$ fixes $E$ pointwise.

For (iii), put
\[
f:=\frac{x-a}{x-d},\qquad E:=K(y),\qquad y^\ell=f.
\]
Since $v_a(f)=1$, the element $f$ is not an $\ell$-th power
in $K$. The Kummer facts recalled in Section~2.1 therefore
show that $E/K$ is cyclic of degree $\ell$ and is the
splitting field of $P(T):=T^\ell-f$.

Let $v$ be either $v_a$ or $v_d$, and let $w$ be a
prolongation to $E$, normalized by $w(E^\times)=\Z$.
The valuation identity from Section~2.1 gives
\[
\ell w(y)=w(f)=e(w/v)v(f).
\]
Since $v_a(f)=1$ and $v_d(f)=-1$, we obtain
$\ell\mid e(w/v)$. As $e(w/v)\leq [E:K]=\ell$,
it follows that $e(w/v)=\ell$. The degree formula then
shows that there is a unique prolongation at each of
$a$ and $d$, so both points are totally ramified.

At every other point, including infinity, we have $v(f)=0$.
Thus $P\in\mathcal O_v[T]$ and
\[
\operatorname{disc}(P)=\pm\ell^\ell f^{\ell-1}
\]
is a unit in $\mathcal O_v$, since $\ell\ne\operatorname{char}(k)$.
The discriminant criterion from Section~2.1 shows that
$E/K$ is unramified at every such point.

Finally, the inertia group above $a$ is the whole group
$\operatorname{Gal}(E/K)$. In particular, it contains the
automorphism sending $y$ to $\zeta y$.
By the surjectivity of restriction on inertia groups recalled
in Section~2.1, this automorphism lifts to an inertia
automorphism in $\Gal(\Omega/K)$ above $a$.

For (iv), choose a prolongation $v$ of $v_a$ to $\Omega$
such that $g\in I_v$. Fix $j\in\Z$, and put
\[
a_j:=\zeta^{-j}s^j(a),\qquad w:=v\circ s^{-j}.
\]
Since $s(K)=K$ and
$s^j(x-a)=\zeta^j(x-a_j)$, we have
\[
\mathcal O_w=s^j(\mathcal O_v),\qquad
\mathcal O_w\cap K
=s^j\bigl(k[x]_{(x-a)}\bigr)
=k[x]_{(x-a_j)}.
\]
Thus $w$ prolongs $v_{a_j}$.
The conjugate $s^jgs^{-j}$ fixes $K$ and preserves
$\mathcal O_w$. Its residue action is trivial because
$\kappa(w)$ is algebraic over the algebraically closed
field $k$, hence equals $k$.
Therefore $s^jgs^{-j}\in I_w$, as required.
\end{proof}
The following lemma constructs an element fixed by a continuous
automorphism of a profinite group as an ordered product along
an orbit. Recall that a sequence in a profinite group converges
to $x$ if and only if its image in every finite continuous quotient
is eventually equal to the image of $x$.
\begin{lemma}\label{lem:ordered-product}
Let $G$ be a profinite group, and let $s:G\to G$
be a continuous automorphism.
Let $g_0\in G$ and put $g_i=s^i(g_0)$ for $i\in\Z$.
Suppose that, in every finite continuous quotient of $G$, all but
finitely many $g_i$ have trivial image. Then
\[
\eta=\lim_{m\to\infty}P_m,
\qquad P_m=g_{-m}g_{-m+1}\cdots g_m,
\]
exists in $G$ and satisfies $s(\eta)=\eta$.
In any finite continuous quotient in which all $g_i$
for $i\ne0$ are trivial, the image of $\eta$ is the
image of $g_0$.
\end{lemma}

\begin{proof}
In a fixed finite quotient only finitely many factors can be
nontrivial, so the images of $P_m$ eventually stabilize.
These limits are compatible with the quotient maps and hence
define an element $\eta$ of the inverse limit $G$.
The hypothesis also gives
\[
\lim_{m\to\infty}g_m=1=\lim_{m\to\infty}g_{-m}.
\]
Since $s(g_i)=g_{i+1}$, we have
\[
\begin{aligned}
s(P_m)
&=s(g_{-m})s(g_{-m+1})\cdots s(g_m)\\
&=g_{-m+1}\cdots g_mg_{m+1}\\
&=g_{-m}^{-1}P_mg_{m+1}.
\end{aligned}
\]
Continuity and passage to the limit give $s(\eta)=\eta$.
The final assertion follows by computing the product in the
specified finite quotient.
\end{proof}

The proof of Lemma~\ref{lem:ordered-product} has a formal analogy
with Kikyo's use of average types in
\cite[Theorems~3.1 and~3.3]{Kikyo}. There, eventual truth along
an indiscernible orbit yields
a complete type invariant under the automorphism; this permits
extending the automorphism so that it fixes a realization of
that type. Here, eventual stabilization in every finite quotient
yields an element $\eta\in G$ fixed by $s$:
shifting the ordered products changes only their boundary factors,
which tend to the identity. The related proof of
\cite[Theorem~1]{KikyoShelah} instead stabilizes families of
formulas and uses compatible types to extend an automorphism;
it does not explicitly use average types. These model-theoretic
constructions feed into contradictions involving existential
closedness and compactness, while Lemma~\ref{lem:ordered-product}
supplies the commuting automorphism needed in our realization
arguments.

\section{Chatzidakis--Hrushovski argument for two commuting automorphisms}\label{sec:chatzidakis-argument}

In \cite[Section~1.11]{Survey}, Chatzidakis gives a variant of
Hrushovski's proof that the theory of fields with two commuting
automorphisms has no model companion. We recall the argument in a
form that isolates three ingredients used in our proofs:
\begin{enumerate}
  \item[(A)] realization at an infinite orbit;
  \item[(B)]  an obstruction at certain finite
orbits;
  \item[(C)] the construction, in a single model, of elements with
arbitrarily large finite orbits to which the obstruction applies.
\end{enumerate}
These are supplied by Lemmas~\ref{comm:lem:infinite},
\ref{comm:lem:period}, and~\ref{comm:lem:exist}, respectively.

For the first two lemmas, we fix a field with two commuting
automorphisms $(K,\sigma,\tau)$ of characteristic zero containing
a primitive cube root of unity $\omega$ such that
\[
  \sigma(\omega)=\omega^2,\qquad \tau(\omega)=\omega.
\]
Using $\omega$ as a parameter, we define the following formula
with free variable $z$:
\[
  \Phi_{\mathrm{comm}}(z):=
  \exists x\,\exists y\,
  \left[
    y\ne0
    \ \wedge\ y^3=x+z
    \ \wedge\ \tau(y)=\omega y
    \ \wedge\ \sigma(x)=x
    \ \wedge\ \tau(x)=x
  \right].
\]
The condition $y\ne0$ excludes the degenerate witness
$x=-z$ and $y=0$.
We use the same letters for automorphisms and their extensions, and we also fix $b\in K$.

\begin{lemma}[A: realization at an infinite orbit]
\label{comm:lem:infinite}
Suppose that $\tau(b)=b$ and that $b$ has infinite
$\sigma$-orbit. Then there is an extension
\[
  (K,\sigma,\tau)\subseteq(L,\sigma,\tau)
\]
such that $(L,\sigma,\tau)\models\Phi_{\mathrm{comm}}(b)$.
\end{lemma}

\begin{proof}
Adjoin an indeterminate $x$ fixed by both automorphisms.
For each $i\in\Z$, choose $\alpha_i$ with
$\alpha_i^3=x+\sigma^i(b)$, and put
\[
L:=K(x)(\alpha_i:i\in\Z).
\]
The classes of the elements $x+\sigma^i(b)$ are linearly
independent in the $\mathbb F_3$-vector space
$K(x)^\times/(K(x)^\times)^3$. Indeed, the valuation at $x=-\sigma^i(b)$ takes value $1$
on $x+\sigma^i(b)$ and value $0$ on every other member
of this family. Hence, Kummer theory allows us to define extensions of $\sigma$ and $\tau$ to $L$ by
\[
  \sigma(\alpha_i)=\alpha_{i+1},
  \qquad
  \tau(\alpha_i)=\sigma^i(\omega)\alpha_i.
\]
These extensions commute on $K(x)$ and on every $\alpha_i$, because
\[
  \sigma\tau(\alpha_i)
  =\sigma^{i+1}(\omega)\alpha_{i+1}
  =\tau\sigma(\alpha_i).
\]
Thus $(x,\alpha_0)$ witnesses that $(L,\sigma,\tau)\models\Phi_{\mathrm{comm}}(b)$.
\end{proof}
\begin{lemma}[B: the finite-orbit obstruction]
\label{comm:lem:period}
Suppose that $\sigma^n(b)=b$ for some integer $n\ge1$.
If there is an extension
\[
  (K,\sigma,\tau)\subseteq(L,\sigma,\tau)
\]
such that $(L,\sigma,\tau)\models\Phi_{\mathrm{comm}}(b)$,
then $n$ is even. In particular, if $b$ has odd
$\sigma$-period, the formula $\Phi_{\mathrm{comm}}(b)$
cannot hold in any such extension.
\end{lemma}

\begin{proof}
Let $x,y\in L$ satisfy the conditions in
$\Phi_{\mathrm{comm}}(b)$. Since
$\sigma^n(x+b)=x+b$ and $y^3=x+b$, there is
$j\in\{0,1,2\}$ such that
\[
  \sigma^n(y)=\omega^j y.
\]
Commutativity gives
\[
  \omega^{j+1}y
  =\tau\sigma^n(y)
  =\sigma^n\tau(y)
  =\omega^{2^n+j}y.
\]
As $y\ne0$ and $\omega$ has order $3$, we obtain
$2^n\equiv1\pmod3$, so $n$ is even.
\end{proof}
Put
\[
J:=\{n\in\N\mid n\ge3\ \ \text{and $n$ is odd}\}.
\]

\begin{lemma}[C: existence of finite orbits]
\label{comm:lem:exist}
There is a characteristic-zero difference field
$(K_0,\sigma,\tau)$ containing a primitive cube
root of unity $\omega$ such that
\[
\sigma(\omega)=\omega^2,\qquad \tau=\id_{K_0},
\]
and, for every $n\in J$, there is $b_n\in K_0$ having exact
$\sigma$-period $n$.
\end{lemma}

\begin{proof}
Choose a primitive cube root of unity $\omega$ and a family
of algebraically independent elements
\[
\{b_{n,j}\mid n\in J,\ 0\le j<n\}
\]
over $\Q$. Let us set
\[
K_0:=\Q(\omega)(b_{n,j}\mid n\in J,\ 0\le j<n)
\]
 and let $\tau=\id_{K_0}$. Extend the automorphism
$\omega\mapsto\omega^2$ of $\Q(\omega)$ to $K_0$ by setting
\[
\sigma(b_{n,j})=
\begin{cases}
b_{n,j+1},&0\le j<n-1,\\
b_{n,0},&j=n-1.
\end{cases}
\]
These assignments define an automorphism because they
permute an algebraically independent generating family.
Since $\tau$ is the identity, $\sigma$ and $\tau$ commute.
For each $n\in J$, put $b_n:=b_{n,0}$. The elements
$b_{n,0},\ldots,b_{n,n-1}$ are distinct, so $b_n$ has exact
$\sigma$-period $n$. Moreover, $\tau(b_n)=b_n$ for every
$n\in J$, as required.
\end{proof}

\begin{proof}[Proof of non-existence of a model companion.]
Let $S$ be the theory of fields with two commuting automorphisms,
and suppose that it has a model companion $S^*$.
The theory $S$ is inductive, so the models of $S^*$ are precisely
 existentially closed models of $S$.

Take $(K_0,\sigma,\tau)$ given by Lemma \ref{comm:lem:exist} and embed it into
$(M,\sigma,\tau)\models S^*$.
Lemma~\ref{comm:lem:period} gives
\[
  M\models\neg\Phi_{\mathrm{comm}}(b_n)
  \qquad(n\in J).
\]
Let $\mathcal U$ be a nonprincipal ultrafilter on $J$, and let us define the corresponding ultraproduct and its element
\[
  (N,\sigma,\tau):=(M,\sigma,\tau)^J/\mathcal U,
  \qquad b:=[b_n]_{\mathcal U}\in N.
\]
By \L o\'s's theorem,
\[
  N\models S^*,\qquad
  \tau(b)=b,\qquad
  N\models\neg\Phi_{\mathrm{comm}}(b).
\]
For each $m>0$ and each $n\in J$ with $n>m$, we have
$\sigma^m(b_n)\ne b_n$.
Hence, by \L o\'s's theorem again,
 $b$ has infinite $\sigma$-orbit.
Lemma~\ref{comm:lem:infinite} supplies an extension of $(N,\sigma,\tau)$
 satisfying $\Phi_{\mathrm{comm}}(b)$, which
contradicts the existential closedness of $N$.
\end{proof}

\begin{remark}
Our two main arguments follow this scheme (as suggested in \cite[Remark~1.12]{Survey}), but they additionally require the following:
\begin{equation}
\tag{D}
\text{\emph{The realizing extension should preserve the prescribed theory of the reduct.}}
\end{equation}
In both settings, constructing the realizing extension within
the prescribed class is the additional issue highlighted in \cite[Remark~1.12]{Survey}. Using the notation of Sections~\ref{sec:bounded}
and~\ref{sec:acfa}, the ingredients are as follows.

\begin{itemize}
\item For bounded PAC fields, Lemma~\ref{b:lem:infinite}
supplies (A), subject to its hypotheses on the fixed elementary
base. Lemma~\ref{b:lem:period} supplies (B), excluding
$\Phi_{\mathrm{PAC}}(b)$ whenever $\sigma^n(b)=b$ and
$\gcd(n,q)=1$. Lemma~\ref{b:lem:periods} supplies (C):
it constructs a single model containing the fixed elementary
base and elements of exact $\sigma$-period $\ell^m$ for every
$m\ge1$.  Requirement (D) is ensured by Lemma~\ref{b:lem:completion},
which produces an elementary PAC extension while preserving
the prescribed automorphism action.

\item For a fixed completion $T$ of $\ACFA$,
Proposition~\ref{a:prop:realization} supplies (A) under its
orbit and fixed-point hypotheses, and
Lemma~\ref{a:lem:obstruction} supplies (B) for periods
$n\equiv1\pmod r$.
Lemma~\ref{a:lem:obstruction_exists} supplies (C) by constructing
$\tau$-fixed elements of these periods whose $\sigma$-cycles
are algebraically independent over $A_0$.
This independence ensures the required orbit conditions
in the ultrapower.
Finally, Lemma~\ref{a:lem:acfa} supplies (D), preserving
the completion $T$ of the first difference-field reduct.
\end{itemize}
\end{remark}

\section{Automorphisms of bounded PAC fields}\label{sec:bounded}

In this section we prove Theorem~\ref{b:thm:main}, following
the scheme of Section~\ref{sec:chatzidakis-argument}.
After naming suitable constants to obtain a model-complete
field theory, we use a single Kummer coset to establish an
obstruction at certain finite periods. The main construction
realizes the corresponding existential condition at an
infinite orbit: commuting lifts and a bilateral product of
inertia elements produce the required automorphism action,
and a PAC completion preserves the prescribed elementary
field theory, supplying property~(D).
We also construct elements of arbitrarily large finite
periods to which the obstruction applies.
An ultrapower then turns these finite orbits into an
infinite orbit, yielding the contradiction to the existence
of a model companion.

Let us fix a field $F$ and a prime $\ell$ such that
\begin{equation}\label{b:eq:hyp}
F\text{ is bounded and PAC},\qquad
\operatorname{char}(F)=0,\qquad \mu_\infty\subseteq F,
\qquad \Gal(F)\twoheadrightarrow C_\ell.
\end{equation}
Put $T=\Th_{L_{\mathrm{ring}}}(F)$, where
$L_{\mathrm{ring}}=\{0,1,+,-,\cdot\}$. Since $\mu_\ell\subseteq F$,
Kummer theory makes the last hypothesis equivalent to the existence of
an element
\begin{equation}\label{b:eq:c}
c\in F^\times\setminus(F^\times)^\ell.
\end{equation}
For a tuple $\bd$ from $F$, let $T_{\bd}=\Th(F,\bd)$ in the language
with names for its entries, and write
\[
S_{\bd}=(T_{\bd})_\sigma
=T_{\bd}\cup\{\sigma\text{ is a field automorphism and}\ \sigma(\bd)=\bd\}.
\]
We use the following three standard facts about PAC fields:
\begin{itemize}
  \item absolute Galois groups of PAC
fields are projective;
  \item algebraic extensions of PAC fields are PAC;
  \item a regular inclusion $K\subseteq L$ of perfect PAC fields is elementary if the
restriction map $\Gal(L)\to\Gal(K)$ is an isomorphism.
\end{itemize}
References are \cite[Corollary 11.2.5, Theorem 11.6.2, Corollary 20.3.4]{FJ}; see also \cite[Section 4.3]{CP} and the proof of \cite[Theorem 3.25]{HK}.

The base-change assertion in the following lemma is
\cite[Section~1.4, Fact~3]{BCGeometric}.
We include its proof and record the resulting
commuting lift of an automorphism.
\begin{lemma}\label{b:lem:basechange}
If $M\preccurlyeq K$ are models of the theory of a bounded
characteristic-zero field, then, with compatible algebraic closures,
\[
K^{\alg}=KM^{\alg},\qquad
\Gal(K)\xrightarrow{\sim}\Gal(M).
\]
In particular, if $\sigma\in\Aut(K/M)$, there is
$\rho\in\Aut(K^{\alg})$ extending $\sigma$, fixing $M^{\alg}$
pointwise, and commuting with every element of $\Gal(K)$.
\end{lemma}

\begin{proof}
Elementarity makes $K/M$ regular. Fix a positive integer $n$. There
are finitely many Galois extensions of $M$ of degree at most $n$;
choose primitive-element polynomials for all of them. The assertion
that these polynomials exhaust the Galois extensions of degree at
most $n$, up to isomorphism, is first-order in their coefficients.
One can express it using the field algebras associated to irreducible
polynomials and the existence of a root of one such polynomial in
another associated algebra.

By elementarity the same list exhausts the corresponding extensions
of $K$. Regularity preserves the degrees and the Galois groups of
the listed extensions. Every finite Galois extension of $K$ therefore
lies in $KM^{\alg}$, giving $K^{\alg}=KM^{\alg}$; regularity gives that the restriction map $\Gal(K)\to \Gal(M)$ is an isomorphism.

The fields $K$ and $M^{\alg}$ are linearly disjoint over $M$. Hence
$\sigma$ on $K$ and the identity on $M^{\alg}$ define $\rho$ on their
compositum, which commutes with every element of $\Gal(K)$.
\end{proof}

\begin{proposition}\label{b:prop:constants}
There is a countable tuple $\bd$ from $F$ for which $T_{\bd}$ is
model-complete. Both $T_{\bd}$ and $S_{\bd}$ are then inductive.
\end{proposition}

\begin{proof}

By the model-completeness result for bounded PAC fields
\cite[Section~4.6]{CP}, there is a countable tuple $\bd$ from $F$
such that $T_{\bd}$ is model-complete; see also the paragraph
immediately preceding \cite[Theorem~4.10]{HK}.
It is clear that both $T_{\bd}$ and $S_{\bd}$ are inductive.
\end{proof}

From now on choose $\bd$ as in
Proposition~\ref{b:prop:constants}, enlarged to contain $c$ from
\eqref{b:eq:c} and every element of the countable group
$\mu_\infty$. Naming these additional constants preserves
model completeness, so $T_{\bd}$ remains model-complete and
both $T_{\bd}$ and $S_{\bd}$ are inductive.
Fix $(M,\bd)\models T_{\bd}$.
The constructions will start with $M\preccurlyeq K$ and
$\sigma|_M=\mathrm{id}$. Pointwise fixation of this whole base is a
choice of models used in the proof, not an additional axiom of $S_{\bd}$.

We choose a prime $q\ne\ell$, a primitive $q$-th root of unity
$\zeta\in M$ and consider the following existential formula
\[
\Phi_{\mathrm{PAC}}(z):=\exists x\,\exists y\,\bigl[
\sigma(x)=\zeta x
\ \land\ (x-z)(\zeta x-z)\ne0
\land\ x-z=c(\zeta x-z)y^\ell
\bigr].
\]
\begin{remark}\label{meaning}
Let us note here the meaning of $\Phi_{\mathrm{PAC}}$. For $(K,\bd,\sigma)\models S_{\bd}$ and $b\in K$, we have that
\[
(K,\bd,\sigma)\models \Phi_{\mathrm{PAC}}(b)
\]
if and only if there is $x\in K$ such that the following hold:
\begin{itemize}
  \item $(x-b)\neq 0\neq (\zeta x-b)$;

  \item $\sigma(x)=\zeta x$;

  \item $(x-b)/(\zeta x-b)\in c(K^\times)^\ell$: the fixed non-trivial coset of $(K^\times)^\ell$.
\end{itemize}
\end{remark}
For a related use of multiplicative power cosets to constrain
automorphisms of pseudofinite fields, see
\cite[Section~4]{BHAlgClosure}.

The result below corresponds to (B) from Section \ref{sec:chatzidakis-argument}.
\begin{lemma}\label{b:lem:period}
Let $(K,\bd,\sigma)\models S_{\bd}$, $b\in K$,
and $n\ge1$. If $\sigma^n(b)=b$
and $\gcd(n,q)=1$, then
$(K,\bd,\sigma)\models\neg\Phi_{\mathrm{PAC}}(b)$.
\end{lemma}

\begin{proof}Suppose, towards a contradiction, that
\[
(K,\bd,\sigma)\models\Phi_{\mathrm{PAC}}(b),
\]
and let $x,y\in K$ be corresponding witnesses. Let us define
\[
f:=(x-b)/(\zeta x-b)=cy^\ell.
\]
Since we have
\[
\sigma^n(b)=b,\qquad \sigma^n(\zeta)=\zeta,\qquad \sigma^n(c)=c,\qquad \sigma(x)=\zeta x,
\]
we obtain
\[
\prod_{j=0}^{q-1}\sigma^{jn}(f)
=\prod_{j=0}^{q-1}
  \frac{\zeta^{jn}x-b}{\zeta^{jn+1}x-b}
=1.
\]
The last equality holds because multiplication by $n$ permutes the
residue classes modulo $q$, hence we have
\[
\{\zeta^{jn}\mid j=0,1,\ldots,q-1\}=\{\zeta^{jn+1}\mid j=0,1,\ldots,q-1\}.
\]
On the other hand, we obtain
\[
1=\prod_{j=0}^{q-1}\sigma^{jn}(f)
=\prod_{j=0}^{q-1}c\sigma^{jn}\left(y^\ell\right)
=c^q\left(\prod_{j=0}^{q-1}\sigma^{jn}(y)\right)^\ell.
\]
Therefore, $c^q\in (K^\times)^\ell$. Since $\gcd(q,\ell)=1$, we get $c\in (K^\times)^\ell$, which is a contradiction.
\end{proof}

The following construction is the main additional ingredient. It originates from the idea to regard a pseudofinite field as the field of fixed points of a model of ACFA and then to try to use Hrushovski's proof for two commuting automorphisms in the case of pseudofinite fields.
\begin{lemma}\label{b:lem:free}
Let $I$ be nonempty and finite or countable. The theory of fields with
an $I$-indexed family of automorphisms has a model companion $\operatorname{ACFA}_I$.
Its underlying fields are algebraically closed. An algebraically closed
invariant subfield is an amalgamation base, and its automorphisms
commuting with all the distinguished automorphisms are partial
elementary maps of models of $\operatorname{ACFA}_I$.

If $(C,(\tau_i)_{i\in I})\models \operatorname{ACFA}_I$ is sufficiently saturated,
then $N=\bigcap_{i\in I}\Fix_C(\tau_i)$ is PAC.
\end{lemma}

\begin{proof}
For finite $I$, existence of the model companion in characteristic
zero follows from \cite[Corollary~4.7]{MS}. The resulting operators are
automorphisms with no imposed relations, and the underlying
fields of existentially closed models are algebraically closed
by \cite[Theorem~4.6]{MS}. For the corresponding result in
positive characteristic, see \cite[Section~4.2]{BHKK}.

An existentially closed field with finitely many distinguished free
generators remains existentially closed on deleting some generators.
Indeed, first take an algebraically closed extension witnessing a
finite-reduct existential condition. Each omitted automorphism of the
algebraically closed base extends independently to this larger
algebraically closed field. Existential closedness in the full language
then gives the required witness.

For countable $I$, let $\operatorname{ACFA}_I$ be the union of the finite-reduct
companions. Every finite subset of these axioms and the quantifier-free
diagram of an $I$-field is realized by embedding the appropriate finite
reduct into its companion. Compactness gives an embedding of the full
$I$-field into a model of $\operatorname{ACFA}_I$. An embedding between two such models
is elementary, since each formula uses only finitely many distinguished
automorphisms. Thus $\operatorname{ACFA}_I$ is the required companion.

Over an algebraically closed invariant field $\Omega$, two extensions
can be amalgamated using the fraction field of their tensor product
over $\Omega$. The tensor product is a domain, and each distinguished
automorphism acts on both factors. There are no relations to check.
Amalgamation and model-completeness show that $\operatorname{ACFA}_I$ together with the
quantifier-free diagram of $\Omega$ is complete. Hence an automorphism
of this $I$-field $\Omega$ is partial elementary in a model of $\operatorname{ACFA}_I$.

Finally, let $V/N$ be an absolutely irreducible affine variety.
For every finite nonempty $J\subseteq I$, the $J$-reduct of $C$
is existentially closed as a field with an action of the free
group on $J$. By \cite[Proposition~2.4]{BKTorsion}, its common
fixed field
\[
N_J:=\bigcap_{j\in J}\Fix_C(\tau_j)
\]
is PAC. Since $N\subseteq N_J$, the variety $V$ has an
$N_J$-rational point. Thus the type asserting membership in
$V$ and fixation by every $\tau_i$, $i\in I$, is finitely
satisfiable in $C$. As $I$ is finite or countable, saturation
gives a realization of this type. This is an $N$-rational
point of $V$, proving that $N$ is PAC.
\end{proof}
The result below corresponds to (D) from Section \ref{sec:chatzidakis-argument}.
\begin{lemma}\label{b:lem:completion}
Let $K$ be a PAC field of characteristic zero, and let
$(\gamma_i)_{i\in I}$ generate a dense subgroup of $\Gal(K)$,
where $I$ is nonempty and finite or countable. Let $\Omega$
be an algebraically closed field containing a fixed algebraic
closure $K^{\alg}$ of $K$. Suppose that $s,\tau_i\in\Aut(\Omega)$
satisfy
\[
s(K)=K,\qquad s\tau_i=\tau_i s,\qquad
\tau_i|_{K^{\alg}}=\gamma_i
\quad(i\in I).
\]
Put
\[
F_0:=\bigcap_{i\in I}\Fix_\Omega(\tau_i).
\]
Then there are a field extension $F_0\subseteq L$ and an
automorphism $\widetilde\sigma\in\Aut(L)$ such that
\[
K\preccurlyeq L
\qquad\text{and}\qquad
\widetilde\sigma|_{F_0}=s|_{F_0}.
\]
In particular, $\widetilde\sigma|_K=s|_K$.
\end{lemma}

\begin{proof}
Each $\tau_i$ fixes $K$ pointwise, so $K\subseteq F_0$.
The commutation relations give $s(F_0)=F_0$. We embed $(\Omega,(\tau_i)_{i\in I})$ into a sufficiently saturated
and strongly homogeneous model
\[
(C,(\tau_i)_{i\in I})\models\operatorname{ACFA}_I.
\]
By Lemma~\ref{b:lem:free}, the automorphism
$s$ of the algebraically closed $I$-field $\Omega$ is partial
elementary. It therefore extends to an automorphism $S$ of
this entire structure. In particular,
\[
S\tau_i=\tau_iS\qquad(i\in I).
\]

Set
\[
N:=\bigcap_{i\in I}\Fix_C(\tau_i).
\]
By Lemma~\ref{b:lem:free}, $N$ is PAC. Moreover,
$F_0\subseteq N$ and $S(N)=N$. Since the $\gamma_i$
topologically generate $\Gal(K)$, we have
\[
N\cap K^{\alg}
=\bigcap_{i\in I}\Fix_{K^{\alg}}(\gamma_i)
=K.
\]
Thus $N/K$ is regular. Let $N^{\alg}$ denote the relative
algebraic closure of $N$ in $C$; it contains $K^{\alg}$.
The restriction map is consequently surjective:
\[
r:\Gal(N)\twoheadrightarrow\Gal(K).
\]

Since $K$ is PAC, $\Gal(K)$ is projective. Choose a continuous
section $j:\Gal(K)\to\Gal(N)$ of $r$, and put
\[
H:=j(\Gal(K)),\qquad L:=(N^{\alg})^H.
\]
The subgroup $H$ is closed. The extension $L/N$ is algebraic,
so $L$ is PAC.

Each $\tau_i$ fixes $N$ pointwise and preserves $N^{\alg}$.
Their restrictions to $N^{\alg}$ topologically generate
$\Gal(N)$, since their common fixed field there is exactly $N$.
Also, $S(N)=N$ implies that $S$ preserves $N^{\alg}$ and
induces a continuous conjugation automorphism of $\Gal(N)$.
As $S$ commutes with every $\tau_i$, it centralizes their
closed generated subgroup:
\[
SgS^{-1}=g\qquad\text{for $g\in\Gal(N)$}.
\]
In particular, $S(L)=L$. Define
$\widetilde\sigma:=S|_L$.

We now verify that $K\subseteq L$ is elementary.
Since $r(H)=\Gal(K)$,
\[
L\cap K^{\alg}
=(K^{\alg})^{r(H)}
=(K^{\alg})^{\Gal(K)}
=K.
\]
Thus $L/K$ is regular. Moreover, $N^{\alg}$ is also an
algebraic closure of $L$, and restriction gives an isomorphism
\[
\Gal(L)=H\xrightarrow{\ \sim\ }\Gal(K).
\]
The elementary-embedding theorem for perfect PAC fields
therefore yields $K\preccurlyeq L$.

Finally, $F_0\subseteq N\subseteq L$, and $S$ extends $s$ on
$\Omega$. Hence $\widetilde\sigma|_{F_0}=s|_{F_0}$, as required.
\end{proof}

The next result about increasing  finite periods corresponds to (C) from Section \ref{sec:chatzidakis-argument}.
\begin{lemma}\label{b:lem:periods}
There is $(K_0,\bd,\sigma)\models S_{\bd}$ such that
\begin{itemize}
  \item $M\preccurlyeq K_0$;

  \item $\sigma|_M=\mathrm{id}$;

  \item for all $m\ge1$, there is $b_m\in K_0$ such that
  \[
  \left|\left\{\sigma^i(b_m)\mid i\in \mathbb{Z}\right\}\right|=\ell^m,
  \]
  that is $b_m$ has exact $\sigma$-period $\ell^m$.
  \end{itemize}
\end{lemma}

\begin{proof}
Choose a nonempty finite or countable family $(\gamma_i)_{i\in I}$
of topological generators of $\Gal(M)$.
Since $\mu_\infty\subseteq M$, every $\gamma_i$ fixes all
roots of unity in $M^{\alg}$.
Apply Lemma~\ref{a:lem:hahn}\textup{(ii)} with
$k=M^{\alg}$ and $\rho=\mathrm{id}$.
This gives an algebraically closed field $\Omega$ containing
$M^{\alg}$, commuting lifts $s,d_i$ with
\[
s|_{M^{\alg}}=\mathrm{id},\qquad
d_i|_{M^{\alg}}=\gamma_i,
\]
and elements $b_m$ in
\[
F_0:=\bigcap_{i\in I}\Fix_\Omega(d_i)
\]
having exact $s$-period $\ell^m$ for every $m\ge1$.
Lemma~\ref{b:lem:completion}, applied with $K=M$,
gives $M\preccurlyeq K_0$ and an automorphism $\sigma$ of
$K_0$ extending $s|_{F_0}$.
Thus $\sigma$ fixes $M$ pointwise and preserves the exact
periods of all the $b_m$. Since $\bd\subseteq M$,
we have $(K_0,\bd,\sigma)\models S_{\bd}$.
\end{proof}

Below is the crucial result about realizing the coset at an infinite orbit which corresponds to (A) from Section \ref{sec:chatzidakis-argument}. The main difficulty still comes from combining with the corresponding (D) which is Lemma~\ref{b:lem:completion}.

\begin{lemma}\label{b:lem:infinite}
Suppose $(M,\bd)\preccurlyeq(K,\bd)\models T_{\bd}$,
$\sigma\in\Aut(K/M)$, and $b\in K$ has infinite $\sigma$-orbit.
There is an extension $(L,\bd,\widetilde\sigma)$ of
$(K,\bd,\sigma)$ in $S_{\bd}$ such that
\[
(K,\bd)\preccurlyeq(L,\bd)\qquad
\text{and}\qquad (L,\bd,\widetilde\sigma)\models\Phi_{\mathrm{PAC}}(b).
\]
\end{lemma}

\begin{proof}
By Lemma~\ref{b:lem:basechange}, we have
$K^{\alg}=KM^{\alg}$ and $\sigma$ extends to
$\rho\in\Aut(K^{\alg})$ fixing $M^{\alg}$ pointwise and centralizing
$\Gal(K)$. Let us choose $\beta\in M^{\alg}$ with $\beta^\ell=c$ and $\omega\in \mu_\ell\setminus \{1\}$.

Since $K$ is bounded, $\Gal(K)$ is small and has a finite or
countable family $(\gamma_i)_{i\in I}$ of topological generators.
We first construct suitable commuting lifts.

\begin{claim}\label{cl1}
There are an element $x$ transcendental over $K$, an algebraic
closure $\Omega$ of $K(x)$ containing $K^{\alg}$, an element
$\alpha\in\Omega$, and automorphisms $s,d_i\in\Aut(\Omega)$ such that
\[
\alpha^\ell=\frac{x-b}{\zeta x-b},
\]
and, for every $i\in I$,
\[
s|_{K^{\alg}}=\rho,\qquad
d_i|_{K^{\alg}}=\gamma_i,\qquad
s(x)=\zeta x,\qquad d_i(x)=x,\qquad
sd_i=d_is,\qquad d_i(\alpha)=\alpha.
\]
\end{claim}

\begin{claimproof}
Since $b$ has infinite $\sigma$-orbit, we have $b\ne0$.
Every $\gamma_i$ fixes all roots of unity in $K^{\alg}$,
since these belong to $M\subseteq K$.
Apply Lemma~\ref{a:lem:hahn}\textup{(i)} with
$k=K^{\alg}$ and series variable $x$, choosing
$\epsilon(1)=\zeta$.
We obtain an algebraic closure $\Omega$ of $K(x)$
containing $K^{\alg}$ inside
\[
A:=\bigcup_{m\ge1}K^{\alg}((x^{1/m})),
\]
together with automorphisms $s,d_i$ of $A$ restricting to $\Omega$ and satisfying
\[
s|_{K^{\alg}}=\rho,\qquad
d_i|_{K^{\alg}}=\gamma_i,\qquad
s(x)=\zeta x,\qquad d_i(x)=x,\qquad sd_i=d_is.
\]
Here each $d_i$ is induced by the coefficientwise action
of $\gamma_i$ on $A$.

Since $b\ne0$, we have
\[
R(x):=\frac{x-b}{\zeta x-b}
     =\frac{1-x/b}{1-\zeta x/b}
     \in1+xK[[x]]\subset A.
\]
The reduction modulo $x$ of $Y^\ell-R(x)$ is
$Y^\ell-1$, which has the simple root $1$ because
$\charac(K)=0$.
Hensel's lemma therefore gives a unique
$\alpha\in1+xK[[x]]$ such that
\[
\alpha^\ell=R(x)=\frac{x-b}{\zeta x-b}.
\]
This $\alpha$ is algebraic over $K(x)$ and hence belongs
to $\Omega$. Finally, every $d_i$ fixes $K[[x]]$
pointwise, so $d_i(\alpha)=\alpha$ for all $i\in I$.
\end{claimproof}
To realize $\Phi_{\mathrm{PAC}}(b)$, we need the corrected
lifts to fix $y:=\alpha/\beta$. Write
\[
d_i(\beta)=\omega^{a_i}\beta,
\qquad \text{where }a_i\in\{0,\ldots,\ell-1\}.
\]
Since $d_i(\alpha)=\alpha$, we have
$d_i(y)=\omega^{-a_i}y$.
Claim~\ref{cl2} provides an automorphism
$\eta\in\Gal(\Omega/K^{\alg}(x))$ commuting with $s$
and satisfying $\eta(\alpha)=\omega\alpha$.
The corrected lifts $d_i\eta^{a_i}$ will then fix $y$
and still commute with $s$.

Choose $\lambda\in K^{\alg}$ with $\lambda^\ell=\zeta$.
Then
\[
(\lambda\alpha)^\ell
=\frac{x-b}{x-\zeta^{-1}b},
\qquad
K^{\alg}(x,\lambda\alpha)=K^{\alg}(x,\alpha).
\]
Since $b\ne0$ and $\zeta\ne1$, the points $b$ and
$\zeta^{-1}b$ are distinct.
Lemma~\ref{a:lem:ramification}(iii), applied to
$\lambda\alpha$, therefore shows that
$K^{\alg}(x,\alpha)/K^{\alg}(x)$ is cyclic of degree $\ell$,
totally ramified at $b,\zeta^{-1}b$, and unramified elsewhere.
\begin{claim}\label{cl2}
There is $\eta\in \Gal(\Omega/K^\alg(x))$ such that
\[
s\eta s^{-1}=\eta\qquad \text{and}\qquad \eta(\alpha)=\omega\alpha.
\]
\end{claim}
\begin{claimproof}
Let us set $G:=\Gal(\Omega/K^\alg(x))$. Since $b$ is a point of total ramification, by the ``Moreover'' part of Lemma \ref{a:lem:ramification}(iii), there is  an inertia element $g_0\in G$ above $b$ with
$g_0(\alpha)=\omega\alpha$. For $k\in\Z$, set
\[
g_k:=s^kg_0s^{-k},\qquad b_k':=\zeta^{-k}\sigma^k(b).
\]
By Lemma~\ref{a:lem:ramification}(iv), $g_k$ is an inertia
element above $b_k'$.

\smallskip
\noindent\textbf{Subclaim.}
The points $b_k'$ are pairwise distinct, and
$b_k'\notin\{b,\zeta^{-1}b\}$ whenever $k\ne0$.

\smallskip
\noindent\emph{Proof of the Subclaim.}
Since $\zeta\in M$ and $\sigma|_M=\id$, we have
$\sigma(\zeta)=\zeta$.
For $r\in\Z\setminus\{0\}$ and $h\in\Z$, an equality
$\sigma^r(b)=\zeta^h b$ would therefore imply
\[
\sigma^{qr}(b)=\zeta^{qh}b=b,
\]
contrary to the infinite $\sigma$-orbit of $b$.

If $b_i'=b_j'$ with $i\ne j$, then
\[
\sigma^{i-j}(b)=\zeta^{i-j}b,
\]
contradicting the preceding observation.
Likewise, if $b_k'=\zeta^{-\delta}b$ for
$k\ne0$ and $\delta\in\{0,1\}$, then
\[
\sigma^k(b)=\zeta^{k-\delta}b,
\]
which is again impossible.
\hfill$\square_{\text{\tiny Subclaim}}$

\smallskip
By Lemma \ref{a:lem:ramification}(i), every finite Galois extension $E/K^\alg(x)$ in $\Omega$ ramifies at only
finitely many points. Since $K^\alg$ is algebraically closed, inertia at an
unramified point restricts trivially to $E$ by Lemma \ref{a:lem:ramification}(ii). Since the points $b_k'$ are pairwise distinct (see Subclaim), it follows that
\[
g_k|_E=\mathrm{id}\quad\text{for all but finitely many }k.
\]
Apply Lemma~\ref{lem:ordered-product} to the continuous
automorphism
\[
G\longrightarrow G,\qquad g\longmapsto sgs^{-1}.
\]
Its iterates send $g_0$ to $g_k=s^kg_0s^{-k}$.
We obtain the convergent ordered product
\[
\eta=\lim_{r\to\infty}g_{-r}g_{-r+1}\cdots g_r\in G
\]
such that $s\eta s^{-1}=\eta$. The condition
\[
b_k'\notin\{b,\zeta^{-1}b\}\qquad(k\ne0)
\]
from Subclaim ensures that every $g_k$ with $k\ne0$ fixes $\alpha$.
Indeed, the finite Galois extension
$K^{\alg}(x)\subset K^{\alg}(x,\alpha)$ is ramified only at $b$ and $\zeta^{-1}b$.
Since $g_k$ is an inertia element above $b_k'$,
Lemma~\ref{a:lem:ramification}(ii) gives
\[
g_k|_{K^{\alg}(x,\alpha)}=\id
\qquad(k\ne0).
\]
Consequently, every finite product defining $\eta$ acts on
$\alpha$ exactly as $g_0$:
\[
(g_{-r}\cdots g_r)(\alpha)=g_0(\alpha)=\omega\alpha.
\]
Passing to the limit yields $\eta(\alpha)=\omega\alpha$ by Lemma~\ref{lem:ordered-product} again.
\end{claimproof}
Let us set
\[
\tau_i:=d_i\eta^{a_i}.
\]
Then $\tau_i|_{K^{\alg}}=\gamma_i$, each $\tau_i$ fixes $x$, and
$s\tau_i=\tau_i s$. Moreover, we obtain
\[
\tau_i(\alpha)=\omega^{a_i}\alpha,\qquad
\tau_i(\beta)=\gamma_i(\beta)=\omega^{a_i}\beta\qquad \qquad \Longrightarrow\qquad \qquad
\tau_i(\alpha/\beta)=\alpha/\beta.
\]
We finally define $y:=\alpha/\beta$. By Lemma~\ref{b:lem:completion}, we get $K\preccurlyeq L$ and $\widetilde{\sigma}\in \Aut(L)$ such that
\[
x,y\in L,\qquad
\widetilde\sigma(x)=\zeta x,\qquad
\frac{x-b}{\zeta x-b}=\alpha ^\ell=\beta ^\ell\frac{\alpha ^\ell}{\beta ^\ell}=cy^\ell,\qquad \widetilde{\sigma}|_K=\sigma.
\]
By Remark \ref{meaning}, we obtain
\[
(L,\bd,\widetilde{\sigma})\models \Phi_{\mathrm{PAC}}(b).
\]
Elementarity on the
field reduct preserves the full type of $\bd$, hence $(K,\bd)\preccurlyeq(L,\bd)$.
\end{proof}

We can finish now the proof of our first main result.
\begin{proof}[Proof of Theorem~\ref{b:thm:main}]
Proposition~\ref{b:prop:constants} gives model-completeness of $T_{\bd}$
and inductivity of $S_{\bd}$. Suppose $S_{\bd}$ had a model companion
$S_{\bd}^*$. Because $S_{\bd}$ is inductive, the models of $S_{\bd}^*$
would be precisely its existentially closed models, and would
satisfy $S_{\bd}$.

We embed the model from Lemma~\ref{b:lem:periods} into
$(K,\bd,\sigma)\models S_{\bd}^*$. Model-completeness of $T_{\bd}$
makes this embedding elementary on the field reduct. The fixed base
$M$ and the exact periods are preserved. Lemma~\ref{b:lem:period}
gives $\neg\Phi_{\mathrm{PAC}}(b_m)$ for every $m$.

Let us take a nonprincipal ultrapower, with
\[
(K^*,\bd,\sigma)=(K,\bd,\sigma)^{\N_{>0}}/\mathcal U,\qquad
b^*=[(b_m)_m]_{\mathcal U}.
\]
It is a model of $S_{\bd}^*$ satisfying $\neg\Phi_{\mathrm{PAC}}(b^*)$. For each fixed
$k\ge1$, the inequality $\sigma^k(b_m)\ne b_m$ holds eventually,
since $\ell^m>k$. Thus $b^*$ has infinite orbit. The diagonal copy
of $M$ remains elementary and pointwise fixed. Lemma~\ref{b:lem:infinite}
produces an extension in $S_{\bd}$ realizing the existential formula
$\Phi_{\mathrm{PAC}}(b^*)$, contradicting existential closedness.
\end{proof}

\begin{corollary}\label{b:cor:free}
Let $F$ be a PAC field of characteristic zero containing $\mu_\infty$.
If $\Gal(F)$ is free profinite of finite rank $e\ge1$, or free
pro-$p$ of finite rank $e\ge1$, then the conclusions of
Theorem~\ref{b:thm:main} hold.
\end{corollary}

\begin{proof}
A finitely generated profinite group is small, hence $F$ is bounded.
A nontrivial finite-rank free profinite group has a quotient $C_\ell$
for every prime $\ell$; a nontrivial free pro-$p$ group has a quotient
$C_p$. It is enough now to apply Theorem~\ref{b:thm:main}.
\end{proof}

\begin{remark}
Let us note that boundedness has the following three roles in our arguments.
\begin{enumerate}
  \item It allows
us to obtain a model-complete theory by naming countably many
constants, as in Proposition~\ref{b:prop:constants}; this also makes
the associated automorphism theory inductive.

  \item It gives
the elementary base-change identity $K^{\alg}=KM^{\alg}$ for
$M\preccurlyeq K$, which permits the simultaneous commuting
coefficient lifts.

  \item It ensures that the absolute Galois
group has a countable dense generating set, allowing the
countable auxiliary construction. Countability is only a convenience here: the same auxiliary
construction works for arbitrary index sets $I$, using a sufficiently
saturated and homogeneous model of $\operatorname{ACFA}_I$.
\end{enumerate}
\end{remark}

\begin{remark}\label{nop}
The PAC argument does not directly extend to arbitrary PAC fields
of characteristic $p>0$. For an imperfect PAC field $K$, the
principal difficulty is the completion construction: a common
fixed field of automorphisms of an algebraically closed field
is necessarily perfect, since the unique $p$th root of a fixed
element is also fixed. Consequently, the fields produced by
this construction cannot be elementary extensions of $K$;
indeed,
\[
\left(K^{\alg}\right)^{\Gal(K)}=K^{\mathrm{perf}}\quad \text{(the perfect closure of $K$)},
\]
rather than $K$.

There is a separate difficulty when the available cyclic-prime
quotient has order $p$: Kummer theory no longer applies, there
are no nontrivial $p$th roots of unity, and a perfect field has
no nontrivial multiplicative cosets modulo $p$th powers.
Treating this case would require an Artin--Schreier replacement
for the multiplicative obstruction and a corresponding
modification of the period construction.

Finally, Puiseux series are not algebraically closed in positive
characteristic, although this particular issue can be resolved
using the Hahn-field construction of Remark~\ref{pposit}.
For perfect bounded PAC fields containing
$\mathbb F_p^\alg$ and admitting a cyclic quotient of
prime order $\ell\ne p$, the specific obstacles above can be
avoided by replacing Puiseux series with Hahn series and choosing
the rotation order to be a prime
\[
q\notin\{p,\ell\}.
\]
\end{remark}
We finish this section with the following table summarizing a possible usage of the choices of primes $\ell$ and $q$.
\begin{center}
\small
\textbf{Parameter choices for the characteristic-zero PAC theorem.}
\par\smallskip
\renewcommand{\arraystretch}{1.2}
\begin{tabular}{@{}p{0.50\textwidth}cc@{}}
\toprule
Absolute Galois group & Kummer prime $\ell$ & Rotation order $q$\\
\midrule
$\widehat{\Z}$ (pseudofinite case) & $3$ & $2$\\
Free profinite, finite rank $e\ge1$ & $3$ & $2$\\
Free pro-$p$, finite rank, $p$ odd & $p$ & $2$\\
Free pro-$2$, finite rank & $2$ & $3$\\
\bottomrule
\end{tabular}
\end{center}
\section{A second automorphism of ACFA}\label{sec:acfa}
In this section, we show that the road map from Section \ref{sec:chatzidakis-argument} still works if we consider models of ACFA with automorphisms. We fix a completion $T$ of $\ACFA$, write $\sigma$ for the distinguished automorphism in the language of $T$, and define
\[
 S_T:=T\cup\{\text{$\tau$ is a field automorphism and $\tau\sigma=\sigma\tau$}\}.
\]
We first establish the preservation of the prescribed completion,
corresponding to property~(D) from
Section~\ref{sec:chatzidakis-argument}.
\begin{lemma}\label{a:lem:acfa}
Let $(A,\sigma_A)\models T$ and $B$ be an algebraically closed field with commuting automorphisms $\sigma_B,\tau_B$ such that $(A,\sigma_A)\subseteq (B,\sigma_B)$. Then $(B,\sigma_B,\tau_B)$ embeds into $(C,\sigma_C,\tau_C)$, where $\sigma_C\tau_C=\tau_C\sigma_C$ and $(C,\sigma_C)\models T$.
\end{lemma}

\begin{proof}
Embed $(B,\sigma_B)$ into a sufficiently saturated and strongly homogeneous model $(C,\sigma_C)$ of $\ACFA$. Here and below, $\operatorname{Diag}_{\mathrm{qf}}(N)$ denotes
the set of all quantifier-free sentences in the language
expanded by constants for the elements of $N$ that hold in $N$. Since $B$ is algebraically closed and $\sigma_B$ is an
automorphism of $B$, \cite[Theorem~1.3, p.~3008]{CH}
implies that the theory
\[
\ACFA\cup\operatorname{Diag}_{\mathrm{qf}}(B,\sigma_B)
\]
is complete in the difference-field language with constants
for all elements of $B$. Thus the difference-field automorphism $\tau_B$ is a partial elementary map of $(C,\sigma_C)$, and extends to an automorphism $\tau_C$ of that structure. In particular, $\tau_C$ commutes with $\sigma_C$. Since $(A,\sigma_A)\models T$ and
$(A,\sigma_A)\subseteq(C,\sigma_C)$, model-completeness of
$\ACFA$ gives
$(A,\sigma_A)\preccurlyeq(C,\sigma_C)$.
Hence $(C,\sigma_C)\models T$.
 \end{proof}
Let us fix now an odd prime $\ell\neq\charac(T)$ and a primitive $\ell$-th root $\zeta$ in a model $(A_0,\sigma)\models T$. Choose $c_0,c_1\in A_0$ and $\varepsilon\in\{1,-1\}$ by the following rule:
\begin{enumerate}[label=\textup{(\alph*)}]
\item If $\sigma(\zeta)\neq\zeta$, set $c_0=0,c_1=1$ and $\varepsilon=1$.
\item If $\sigma(\zeta)=\zeta$, choose a two-cycle $c_0\neq c_1$ with $\sigma(c_0)=c_1$ and $\sigma(c_1)=c_0$, and set $\varepsilon=-1$.
\end{enumerate}
The two-cycle in (b) exists by existential closedness. Since $\ell$ is odd, both cases give
\begin{equation}\label{a:eq:mismatch}
 \sigma(\zeta)\neq\zeta^\varepsilon.
\end{equation}
With these parameters, we define the following existential formula
\[
\Phi_{\mathrm{ACFA}}(z):=\exists x\,\exists y
 \left[
 \sigma(x)=\tau(x)=x\ \wedge\
 c_0\neq x+z\neq c_1\ \wedge\
 y^\ell=\frac{x+z-c_0}{x+z-c_1}\ \wedge\
 \tau(y)=\zeta y\right].
\]
Retain the parameters chosen above throughout this section.
For the next two results, let $(k,\sigma,\tau)$ be a field with commuting automorphisms
containing the difference subfield of $(A_0,\sigma)$ generated
by $c_0,c_1,\zeta$, and suppose that $\tau$ fixes these parameters.
Let $b\in k$, and put
\[
a:=c_0-b,\qquad d:=c_1-b.
\]
The next proposition is the construction needed to realize an existential formula in an extension, which corresponds to property (A) from Section \ref{sec:chatzidakis-argument}. The proof follows the lines of the proof of Lemma \ref{b:lem:infinite}, so we will be brief.
\begin{proposition}\label{a:prop:realization}
Assume that $(k,\sigma,\tau)\models S_T$ and $\tau(b)=b$.
If $a$ has infinite $\sigma$-orbit not containing $d$, then there is an $S_T$-extension
\[
(k,\sigma,\tau)\subseteq(K,\sigma,\tau)
\]
such that $(K,\sigma,\tau)\models\Phi_{\mathrm{ACFA}}(b)$.
\end{proposition}
\begin{proof}
Let $H:=k((x^{\Q}))$ and $\Omega$ be the relative algebraic closure of $k(x)$
in $H$. By Remark~\ref{pposit}, $\Omega$ is an algebraic closure
of $k(x)$, and the coefficientwise extensions of $\sigma$ and
$\tau$ to $H$ restrict to commuting automorphisms $\sigma_0,\tau_0$
of $\Omega$ fixing $x$. Let us choose $y\in\Omega$ such that
\[
y^\ell=\frac{x+b-c_0}{x+b-c_1}
      =\frac{x-a}{x-d}.
\]
Since $\tau$ fixes $b,c_0,c_1$, $\tau_0$
fixes $y^\ell$. Thus, for some
$j\in\{0,\ldots,\ell-1\}$, we have
\begin{equation}\label{a:eq:tau0y}
\tau_0(y)=\zeta^j y.
\end{equation}
By Lemma~\ref{a:lem:ramification}(iii), the extension
$k(x,y)/k(x)$ is cyclic of degree $\ell$, totally ramified
at $a,d$, and unramified elsewhere. Let us choose an inertia element
\[
\rho\in G:=\Gal(\Omega/k(x))
\]
above $a$ such that $\rho(y)=\zeta y$.
For $i\in\Z$, we set $\rho_i:=\sigma_0^i\rho\sigma_0^{-i}$. 
Since $\sigma_0(x)=x$, Lemma~\ref{a:lem:ramification}(iv)
shows that $\rho_i$ is an inertia element above $\sigma^i(a)$.

Every finite Galois extension $E/k(x)$ inside $\Omega$
ramifies at only finitely many points. Since the points
$\sigma^i(a)$ are pairwise distinct,
Lemma~\ref{a:lem:ramification}(i)--(ii) gives
\[
\rho_i|_E=\id_E
\qquad\text{for all but finitely many }i.
\]
Apply Lemma~\ref{lem:ordered-product} to the continuous
automorphism
\[
G\longrightarrow G,\qquad
g\longmapsto\sigma_0g\sigma_0^{-1},
\]
whose iterates send $\rho$ to $\rho_i$.
We obtain
\[
\eta:=\lim_{m\to\infty}\rho_{-m}\rho_{-m+1}\cdots\rho_m\in G,
\qquad
\sigma_0\eta\sigma_0^{-1}=\eta.
\]
The orbit hypotheses imply that
$\sigma^i(a)\notin\{a,d\}$ whenever $i\ne0$.
Thus every $\rho_i$ with $i\ne0$ fixes $k(x,y)$ pointwise,
and the final assertion of Lemma~\ref{lem:ordered-product}
gives
\begin{equation}\label{a:eq:etay}
\eta(y)=\zeta y.
\end{equation}
Let us set $\tau_1:=\eta^{\,1-j}\tau_0$. 
Since $\eta$ fixes $k(x)$ pointwise, $\tau_1$ extends $\tau$
on $k$ and fixes $x$. Both $\eta$ and $\tau_0$ commute with
$\sigma_0$, so $\tau_1$ also commutes with $\sigma_0$.
Moreover, \eqref{a:eq:tau0y} and \eqref{a:eq:etay} give
\[
\tau_1(y)
=\eta^{\,1-j}(\zeta^j y)
=\zeta^j\zeta^{1-j}y
=\zeta y.
\]
Since $x$ is transcendental over $k$, we also have
$c_0\ne x+b\ne c_1$. Hence
\[
(\Omega,\sigma_0,\tau_1)\models\Phi_{\mathrm{ACFA}}(b).
\]

Finally, Lemma~\ref{a:lem:acfa} embeds this algebraically
closed pair into a model of $S_T$, preserving the formula
$\Phi_{\mathrm{ACFA}}(b)$, which finishes the proof.
\end{proof}
For the remainder of this section, fix $r\ge1$ such that
$\sigma^r$ fixes the tuple $(c_0,c_1,\zeta)$. The result below corresponds to (B) from Section \ref{sec:chatzidakis-argument}.
\begin{lemma}\label{a:lem:obstruction}
For every positive integer $n\equiv1\pmod r$, we have
\[
\sigma^n(b)=b
\qquad\Longrightarrow\qquad
(k,\sigma,\tau)\models\neg\Phi_{\mathrm{ACFA}}(b).
\]
\end{lemma}

\begin{proof}
Suppose $x,y$ witness $\Phi_{\mathrm{ACFA}}(b)$. In particular, $y\neq0$. Applying $\sigma^n$ to its Kummer equation fixes $x,b$ and acts on $c_0,c_1$ as $\sigma$ does. The fraction from the $\Phi_{\mathrm{ACFA}}$-formula is therefore unchanged if $\varepsilon=1$ and inverted if $\varepsilon=-1$. Hence
\[
 \sigma^n(y)=\xi y^\varepsilon\qquad\text{for some }\ \xi\in\mu_\ell(k).
\]
Because $\tau(\zeta)=\zeta$, the automorphism $\tau$ fixes every element of $\mu_\ell(k)$, including $\xi$. Consequently,
\[
 \tau\sigma^n(y)=\xi\zeta^\varepsilon y^\varepsilon,\qquad
 \sigma^n\tau(y)=\sigma^n(\zeta)\xi y^\varepsilon
                 =\sigma(\zeta)\xi y^\varepsilon.
\]
Since $\sigma\tau=\tau\sigma$ and $y\neq0$, we get $\zeta^\varepsilon=\sigma(\zeta)$, contradicting \eqref{a:eq:mismatch}.
\end{proof}
The next result corresponds to (C) from Section \ref{sec:chatzidakis-argument}.
\begin{lemma}\label{a:lem:obstruction_exists}
There is $(K,\sigma,\tau)\models S_T$ extending
$(A_0,\sigma,\id_{A_0})$, with $\tau=\id_K$, such that
for every $n>0$ there is $b_n\in K$ satisfying
$\sigma^n(b_n)=b_n$ and such that
$b_n,\sigma(b_n),\ldots,\sigma^{n-1}(b_n)$
are algebraically independent over $A_0$.
\end{lemma}

\begin{proof}
Let us define
\[
B:=A_0(b_{n,j}:n\geq1,\ 0\leq j<n),
\]
where the $b_{n,j}$ are algebraically independent over $A_0$.  
We extend $\sigma$ from $A_0$ to an automorphism of $B$ by setting
\[
\sigma(b_{n,j})=
\begin{cases}
b_{n,j+1},&j<n-1,\\
b_{n,0},&j=n-1
\end{cases}
\]
and embed $(B,\sigma)$ into a model $(K,\sigma)\models\ACFA$.
Since $(A_0,\sigma)\models T$, model-completeness of $\ACFA$
gives $(K,\sigma)\models T$.

Let us set $\tau=\id_K$ and $b_n=b_{n,0}$ for every $n\geq1$.
Then $\tau$ commutes with $\sigma$, so
$(K,\sigma,\tau)\models S_T$.
Moreover,
\[
\sigma^n(b_n)=b_n,
\qquad
\sigma^j(b_n)=b_{n,j}\quad(0\leq j<n).
\]
The latter elements remain algebraically independent over $A_0$,
as required. In addition, $\tau$ fixes every $b_n$ and all the chosen parameters
from $A_0$.
\end{proof}
\begin{proof}[Proof of Theorem~\ref{a:thm:main}]
We follow the ultrapower argument from
Section~\ref{sec:chatzidakis-argument} and the proof of
Theorem~\ref{b:thm:main}.
Suppose that $S_T$ has a model companion $T_A$.
Since $T$ is model-complete, $S_T$ is inductive, so models
of $T_A$ are existentially closed models of $S_T$.

Take $(K,\sigma,\id_K)$ and the elements $b_n$ from the
construction in Lemma~\ref{a:lem:obstruction_exists},
which supplies (C), and embed this structure into
$(M,\sigma,\tau)\models T_A$.
Put
\[
J:=\{n\ge2:n\equiv1\pmod r\}.
\]
By (B), Lemma~\ref{a:lem:obstruction}, we have
$M\models\neg\Phi_{\mathrm{ACFA}}(b_n)$ for every $n\in J$.

For a nonprincipal ultrafilter $\mathcal U$ on $J$, set
\[
(N,\sigma,\tau):=(M,\sigma,\tau)^J/\mathcal U,
\qquad b:=[b_n]_{\mathcal U}.
\]
Identifying $A_0$ with its diagonal image, \L o\'s's theorem gives
\begin{equation}\label{a:eq:negphi}
N\models T_A,\qquad \tau(b)=b,\qquad
N\models\neg\Phi_{\mathrm{ACFA}}(b).
\end{equation}

For each fixed $m\in\Z\setminus\{0\}$ and every $n\in J$
with $n>|m|$, algebraic independence of the $n$-cycle
over $A_0$ gives
\[
\sigma^m(c_0-b_n)\notin\{c_0-b_n,c_1-b_n\}.
\]
Indeed, either equality would give a nontrivial linear
relation over $A_0$ between $b_n$ and $\sigma^m(b_n)$.
Together with $c_0\ne c_1$, \L o\'s's theorem therefore
shows that $a:=c_0-b$ has infinite $\sigma$-orbit and
$c_1-b$ lies outside that orbit.

Proposition~\ref{a:prop:realization}, supplying (A) together
with the preservation requirement (D), now gives an
$S_T$-extension of $(N,\sigma,\tau)$ satisfying
$\Phi_{\mathrm{ACFA}}(b)$.
Existential closedness forces the same formula to hold
in $N$, contradicting \eqref{a:eq:negphi}.
\end{proof}
\begin{remark}We finish this section with some comments on the scope of the argument
\begin{enumerate}
  \item For some completions of $\ACFA$, $\sigma$ fixes every root of unity. A proof requiring $\sigma(\zeta)\neq\zeta$ would miss these completions. Interchanging $c_0,c_1$ makes $\sigma$ invert the radicand; this replaces the required mismatch by $\sigma(\zeta)\neq\zeta^{-1}$, which holds when $\ell$ is odd and $\sigma(\zeta)=\zeta$.

  \item Finite ramification and trivial restriction of inertia at unramified points are valid more generally. In this proof, however, the coefficient field is algebraically closed because it is the field reduct of a model of $\ACFA$. Algebraic closedness is used for the Hahn construction and for Lemma~\ref{a:lem:acfa}. The proof does not assume that an arbitrary pair of commuting automorphisms extends to an algebraic closure. It uses a particular algebraically closed extension and
modifies one lift by a power of $\eta$.

  \item The theorem concerns a fixed completion of $\ACFA$, so the nonexistence theorem for arbitrary fields with two commuting automorphisms alone is not enough. The distinction and the algebraic-lifting difficulty are discussed in \cite[Section 3]{KikyoSurvey} and \cite[Remark 1.12]{Survey}.
\end{enumerate}
\end{remark}

\section{Model-complete reducts and existentially closed models}
\label{sec:mc-reducts}

The arguments above concern automorphism expansions of prescribed
complete theories. We conclude with some general observations
about completions and existentially closed models.
First, we prove a transfer result for model companions over
model-complete reducts and show that its converse fails even
for automorphism expansions. We then characterize model-completeness
of a complete theory by the existence of a weakly saturated
existentially closed model. Finally, we show that an incomplete
universal theory can have a model-complete completion without
having such a model.

We use the usual general definition: a model companion of a theory $S$ is a
model-complete theory $S^*$ such that every model of either theory
embeds into a model of the other. In particular, we do not assume
$S\subseteq S^*$ in this definition.

\begin{proposition}\label{mc:prop:transfer}
Let $L_0\subseteq L$ be languages, $R$ be a model-complete $L_0$-theory,
and $S$ be an $L$-theory extending $R$. Suppose that $S$ has
a model companion $S^*$. If $\Gamma$ is a set of $L_0$-sentences
such that $S\cup\Gamma$ is consistent, then $S^*\cup\Gamma$ is
a model companion of $S\cup\Gamma$.
\end{proposition}

\begin{proof}
We first show that $S^*\models R$. Given $A\models S^*$, we use the
companion property to obtain the following $L$-embeddings:
\[
  A\subseteq B\subseteq C,
  \qquad B\models S,\qquad C\models S^*.
\]
Since $S^*$ is model-complete, $A\preccurlyeq C$. The
model-complete theory $R$ admits an $\forall\exists$-axiomatization.
Consider such an axiom
\[
\forall\bar x\,\exists\bar y\,\varphi(\bar x,\bar y),
\]
where $\varphi$ is quantifier-free, and let $\bar a$ be an arbitrary tuple
from $A$ of length $|\bar x|$. Since $B\models R$, there is a tuple $\bar b$ from $B$
such that $B\models\varphi(\bar a,\bar b)$. The inclusion $B\subseteq C$ preserves quantifier-free formulas,
so $C\models\varphi(\bar a,\bar b)$ 
and consequently $C\models\exists\bar y\,\varphi(\bar a,\bar y)$.
As $A\preccurlyeq C$ and $\bar a$ lies in $A$, we obtain
$A\models\exists\bar y\,\varphi(\bar a,\bar y)$. Thus $A$ satisfies every axiom of $R$, and hence $A\models R$.

Now embed $M\models S\cup\Gamma$ into $N\models S^*$. Both
$L_0$-reducts satisfy $R$ (since $R\subseteq S$ and $S^*\models R$), so the embedding is elementary on these
reducts. Hence $N\models\Gamma$. Conversely, if
$N\models S^*\cup\Gamma$, embed $N$ into $M\models S$; the
same argument gives $M\models\Gamma$. This proves the mutual
embedding property. Finally, $S^*\cup\Gamma$ is model-complete,
being an extension of a model-complete theory in the same language.
\end{proof}

\begin{corollary}\label{mc:cor:completions}
Let $R$ be model-complete, possibly incomplete. If $R_\tau$ has
a model companion $C$, then $C\cup R'$ is a model companion of
$(R')_\tau$ for every completion $R'$ of $R$. Consequently, if
$(R')_\tau$ has no model companion for some completion $R'$ of
$R$, then $R_\tau$ has no model companion.
\end{corollary}

\begin{proof}
Apply Proposition~\ref{mc:prop:transfer} with $S=R_\tau$ and
$\Gamma=R'$. Consistency holds because every model of $R'$ admits
the identity automorphism.
\end{proof}

In the setting of Theorem~\ref{a:thm:main}, take $R=\ACFA$ in
the language containing $\sigma$. An automorphism $\tau$ of this
reduct is precisely a field automorphism commuting with $\sigma$.
Thus the failure of a model companion over a single completion of
$\ACFA$ would already imply its failure over $\ACFA$ itself;
Theorem~\ref{a:thm:main} establishes the stronger assertion for
every completion. In Theorem~\ref{b:thm:main}, the requirement that
the added automorphism fix $\bd$ is likewise incorporated by
including the named constants in the reduct language.

We show that the implication from Corollary \ref{mc:cor:completions} cannot be reversed. We need a result from \cite{HoffmannDynamics} first.
\begin{lemma}\label{mc:lem:permutations}
The theory of infinite sets with a bijection $\tau$ has a model
companion, axiomatized by requiring infinitely many $\tau$-cycles
of each finite positive length.
\end{lemma}

\begin{proof}
Apply \cite[Example~2.8]{HoffmannDynamics} with $G=\Z$.
A $\Z$-action is equivalent to a single bijection $\tau$,
and the axiom scheme given there is equivalent to requiring
infinitely many $\tau$-cycles of every finite positive length.
\end{proof}

\begin{example}\label{mc:ex:automorphisms}
There is a model-complete theory $T$ 
such that $(T')_\tau$ has a model companion for every completion
$T'$ of $T$, whereas $T_\tau$ has no model companion.
\end{example}

\begin{proof} 
Let $L=\{P_n\mid n\geq 2\}$ be the language consisting of unary predicate symbols, and let
$\theta_n$ be an $L$-sentence saying that the universe has at least $n$ elements.
Let us define
\[
  T:=\{\forall x\,(P_n(x)\leftrightarrow\theta_n)\mid n\geq 2\}.
\]
Thus if $M\models T$, then $P_n(M)$ is full or empty according to the cardinality of the
universe. Hence for an extension $M\subseteq N$ of models of $T$ we have:
\begin{itemize}
  \item if $M$ is finite, then $M=N$;
  \item if $M$ is infinite, then for all $n\geq 2$ we have $P_n(M)=M$ and $P_n(N)=N$, so $M\preccurlyeq N$ by the theory of infinite pure sets.
\end{itemize}
Therefore $T$ is model-complete.

The completions of $T$ are exactly
\[
  T_m=T\cup\{\theta_m,\neg\theta_{m+1}\}\quad 
       \text{for $m\geq 1$},\qquad
  T_\infty=T\cup\{\theta_n:n\geq 2\}.
\]
Every embedding between models of $(T_m)_\tau$ is an isomorphism,
so $(T_m)_\tau$ is already model-complete. The theory
$(T_\infty)_\tau$ has the companion from
Lemma~\ref{mc:lem:permutations}, together with the axioms making
every $P_n$ full.

Suppose that $TA$ were a model companion of $T_\tau$. For any $m\geq 1$, let $A_m\models T_\tau$ be
the $m$-element model of $T$ with $\tau=\id$. There are embeddings
\[
  A_m\hookrightarrow B\hookrightarrow D,
  \qquad B\models TA,\qquad D\models T_\tau.
\]
The composite preserves all cardinality thresholds, so $|D|=m$.
It follows that the first embedding is an isomorphism, and hence
$A_m\models TA$.

For a nonprincipal ultrafilter $\mathcal V$ on the positive integers,
the ultraproduct
\[
  A=\prod_{m\geq 1}A_m/\mathcal V
\]
is an infinite model of $TA$ with every $P_n$ full and $\tau=\id$.
Adjoining a disjoint two-element $\tau$-cycle, and keeping all
predicates full, gives an extension $E\models T_\tau$. Embed $E$
into $B'\models TA$. The induced embedding $A\hookrightarrow B'$
is not elementary, since
\[
  A\models\forall x\,\tau(x)=x,
  \qquad B'\models\exists x\,\tau(x)\neq x.
\]
This contradicts model-completeness of $TA$.
\end{proof}

We now consider existential closedness relative to an arbitrary
theory $T$. The following result was pointed out to us by Udi Hrushovski.
\begin{proposition}\label{mc:prop:saturated-ec}
Let $T$ be a theory and $M$
be an existentially closed model of $T$. Suppose that $M$ is ``weakly saturated'', that is it
realizes every complete type over $\varnothing$ in finitely many
variables consistent with $\Th(M)$. Then $\Th(M)$
is model-complete.
\end{proposition}

\begin{proof}
Put $S=\Th(M)$. Let $A\models S$, let $A\subseteq B\models T$,
and suppose that $B\models\varphi(\bar a)$, where $\varphi$ is
existential and $\bar a$ is a finite tuple from $A$. Let us choose
$\bar a'$ from $M$ such that
\[
  \operatorname{tp}_M(\bar a'/\varnothing)
  =\operatorname{tp}_A(\bar a/\varnothing).
\]
We claim that the theory
\[
T':= T\cup\operatorname{Diag}_{\mathrm{qf}}(M)\cup\{\varphi(\bar a')\}
\]
is consistent.

Let us take 
$\delta(\bar a',\bar m)\in \operatorname{Diag}_{\mathrm{qf}}(M)$. Since
$M\models\exists\bar y\,\delta(\bar a',\bar y)$, equality of
types gives $A\models\exists\bar y\,\delta(\bar a,\bar y)$.
Witnesses in $A$ still satisfy $\delta$ in $B$. Interpreting the
finitely many diagram constants by these witnesses and $\bar a$
therefore realizes the finite diagram conditions together with
$T$ and $\varphi(\bar a')$ in $B$, proving the claim.

A model of $T'$ yields an extension $N\supseteq M$ satisfying
$T$ and $\varphi(\bar a')$. Existential closedness gives
$M\models\varphi(\bar a')$, and equality of types gives
$A\models\varphi(\bar a)$. Thus every $A\models S$ is
existentially closed among models of $T$. In particular, every
embedding between models of $S$ is existential, so Robinson's
test implies that $S$ is model-complete.
\end{proof}

\begin{corollary}\label{mc:cor:saturated-completion}
If $T$ has a weakly saturated model $M$ which is
existentially closed among models of $T$, then $\Th(M)$ is a
model-complete completion of $T$.
\end{corollary}

For complete theories, the preceding implication yields an exact
characterization of model-completeness.

\begin{corollary}\label{mc:cor:complete-case}
If $T$ is a complete theory, then $T$ is model-complete if and only if $T$ has a
weakly saturated model which is existentially closed.
\end{corollary}

\begin{proof}
Suppose first that $T$ is model-complete. Every embedding between
models of $T$ is elementary, so every model of $T$ is existentially
closed among models of $T$. Choosing an $\aleph_0$-saturated
model gives the required model.

Conversely, suppose that $M\models T$ is weakly saturated 
and existentially closed among models of $T$.
By Corollary~\ref{mc:cor:saturated-completion}, $\Th(M)$ is
model-complete. Since $T$ is complete, it is logically equivalent
to $\Th(M)$, and hence is model-complete.
\end{proof}

\begin{remark}
For every infinite cardinal $\kappa$,
Corollary~\ref{mc:cor:complete-case} remains true with
``weakly saturated'' replaced by ``$\kappa$-saturated''.
Indeed, every consistent theory has a $\kappa$-saturated model,
and if $T$ is model-complete, every model of $T$ is existentially
closed among models of $T$. Conversely, $\kappa$-saturation
implies weak saturation.
\end{remark}

Finally, we observe that the converse of Corollary~\ref{mc:cor:saturated-completion}
fails, even when $T$ is universal.

\begin{example}\label{mc:ex:universal}
There is a universal theory $T$ in a countable language,
with no finite models, which has a model-complete completion
but no weakly saturated existentially closed model.
\end{example}

\begin{proof}
Let $L=\{E\}\cup\{c_n\mid n<\omega\}$, where $E$ is a binary predicate symbol and
the $c_n$ are constant symbols. Let $T$ assert that $E$ is an equivalence
relation, that the constants are pairwise distinct, and that
\[
  \forall x\,(E(x,c_n)\longrightarrow x=c_n)
  \qquad \text{for $n<\omega $}.
\]
These are universal axioms, and every model is infinite. The theory
\[
  S=T\cup\{\forall x\,\forall y\,
                  (E(x,y)\leftrightarrow x=y)\}
\]
is a model-complete completion of $T$: it is the theory of an
infinite pure set with distinct named constants and $E$ defined
as equality. Quantifier elimination for infinite pure sets gives
quantifier elimination here, and the prescribed inequalities
between constants decide all quantifier-free sentences.

Now let $M\models T$ be weakly saturated. Consider the
partial type over the empty set
\[
  p(x)=\{x\neq c_n:n<\omega\}
       \cup\{\forall y\,(E(x,y)\longrightarrow y=x)\}.
\]
Every finite subset of $p(x)$ is realized in $M$ by some
$c_m$ whose index is not among those excluded.
By compactness, $p(x)$ extends to a complete type over
$\varnothing$ consistent with $\Th(M)$.
Weak saturation therefore gives a realization $a\in M$
of $p(x)$.

We adjoin a new element $b$ to the class of $a$, leaving all other
classes unchanged. The resulting extension $N$ still
satisfies $T$, but
\[
  N\models\exists y\,(E(a,y)\land y\neq a),
  \qquad
  M\models\neg\exists y\,(E(a,y)\land y\neq a).
\]
Hence $M$ is not an existentially closed model of $T$.
\end{proof}

\end{document}